\documentclass[a4paper,12pt,reqno]{amsart}

\usepackage{amsmath,amssymb,amsthm}
\usepackage{latexsym}
\usepackage{color}
\usepackage{graphicx}
\usepackage{mathrsfs}
\usepackage{enumerate}
\usepackage[abbrev]{amsrefs}
\usepackage[T1]{fontenc}
\usepackage{mathtools}
\mathtoolsset{showonlyrefs=true}

\allowdisplaybreaks[4]

\theoremstyle{plain}
\newtheorem{thm}{Theorem}[section]
\newtheorem{lemm}[thm]{Lemma}

\newtheorem{cor}[thm]{Corollary}

\theoremstyle{definition}

\newtheorem{rem}[thm]{Remark}

\makeatletter

\@addtoreset{equation}{section}
\makeatother

\renewcommand{\div}{\operatorname{div}}
\newcommand{\dB}{\dot{B}}

\newcommand{\supp}{\operatorname{supp}}

\renewcommand{\leq}{\leqslant}
\renewcommand{\geq}{\geqslant}

\newcommand{\n}[1]{{\left\|#1\right\|}}

\newcommand{\lp}[1]{\left[#1\right]}
\newcommand{\Mp}[1]{\left\{#1\right\}}
\renewcommand{\sp}[1]{\left(#1\right)}

\begin{document}
\title[the $3$D Navier--Stokes equations with large rough vertical velocities]
{Global solutions to the Navier--Stokes equations with large vertical velocities in $\dot{B}_{\infty,\sigma}^{-1}(\mathbb{R}^3)$}
\author[M.~Fujii]{Mikihiro Fujii}
\address[M.~Fujii]{Graduate School of Science, Nagoya City University, Nagoya, 467-8501, Japan}
\email[M.~Fujii]{fujii.mikihiro@nsc.nagoya-cu.ac.jp}
%\date{}
%\thanks{}
\begin{abstract}
In this paper, we consider the Cauchy problem for the $3$D incompressible Navier--Stokes equations and prove the existence of unique global solutions in the framework that the horizontal component of the velocity field is small in some critical Besov spaces including the classical Fujita--Kato class, {while the vertical component is large in the wide class} $\dot{B}_{\infty,\sigma}^{-1}(\mathbb{R}^3)$ ($1 \leq \sigma < \infty$) {where the Navier--Stokes equations are known to be ill-posed in}.
\end{abstract}
\maketitle

\section{Introduction}\label{sec:intro}
Let us consider the incompressible Navier--Stokes equations on the three dimensional whole space:
\begin{align}\label{eq:NS-1}
    \begin{cases}
        \partial_t u - \Delta u + (u \cdot \nabla) u + \nabla P = 0, \quad & t>0,x \in \mathbb{R}^3, \\
        \div u= 0, & t \geq 0, x \in \mathbb{R}^3, \\
        u(0,x)=a(x), & x \in \mathbb{R}^3,
    \end{cases}
\end{align}
where $u=(u_1(t,x),u_2(t,x),u_3(t,x)):[0,\infty) \times \mathbb{R}^3 \to \mathbb{R}^3$ and $P=P(t,x):(0,\infty) \times \mathbb{R}^3 \to \mathbb{R}^3$ are the unknown velocity and pressure of the fluid, respectively, whereas $a=(a_1(x),a_2(x),a_3(x)):\mathbb{R}^3 \to \mathbb{R}^3$ denotes the given initial velocity satisfying the compatibility condition $\div a = 0$.
It is well-known that \eqref{eq:NS-1} has the invariant scaling; if $u$ and $p$ are solutions to  \eqref{eq:NS-1}, then $u^{\lambda}(t,x):=\lambda u(\lambda^2 t, \lambda x)$ and $P^{\lambda}(t,x):=\lambda^2 P(\lambda^2 t, \lambda x)$ also solve \eqref{eq:NS-1} for all $\lambda$.
A Banach space $X \subset \mathscr{S}'(\mathbb{R}^3)$ is said to be a scaling critical space if $\n{u^\lambda(0,\cdot)}_X=\n{u(0,\cdot)}_X$ for all $\lambda$.
It was Fujita and Kato \cite{Fuj-Kat-64} who initiated the constructing unique local solutions for large data and global solutions for small data in the critical Sobolev space $\dot{H}^{\frac{1}{2}}(\mathbb{R}^3)$.
Since this pioneering work, 
it has been widely known as the Fujita--Kato principle that finding unique solutions in scaling critical spaces is important.
Kato proved improved the functional framework of \cite{Kat-84} to a wider critical space $L^3(\mathbb{R}^3)$.
Leter, \cites{Can-Pla-96,Pla-98} refined results of \cites{Kat-84} and proved the global well-posedness of \eqref{eq:NS-1} for small data in the scaling critical Besov spaces $\dB_{p,q}^{3/p-1}(\mathbb{R}^3)$ for all $1 \leq p < \infty$ and $1 \leq q \leq \infty$. 
More precisely, there exists a positive constant $\delta=\delta(p,q)$ such that for any $a \in \dB_{p,q}^{3/p-1}(\mathbb{R}^3)$ with $\div a = 0$, \eqref{eq:NS-1} possesses a unique global solution 
\begin{align}\label{CL}
    u \in \widetilde{C}\sp{[0,\infty);\dB_{p,q}^{\frac{3}{p}-1}(\mathbb{R}^3)} \cap \widetilde{L^1}\sp{0,\infty;\dB_{p,q}^{\frac{3}{p}+1}(\mathbb{R}^3)}
\end{align}
and it holds
\begin{align}
    \n{u}_{E_{p,q}(0,\infty)} \leq C\n{u_0}_{\dB_{p,q}^{\frac{3}{p}-1}},
\end{align}
where we have defined
\begin{align}
    \n{u}_{E_{p,q}(T_1,T_2)}
    :=
    \n{u}_{\widetilde{L^{\infty}}(T_1,T_2;\dB_{p,q}^{\frac{3}{p}-1})}
    +
    \n{u}_{\widetilde{L^1}(T_1,T_2;\dB_{p,q}^{\frac{3}{p}+1})},
    \qquad
    0 \leq T_1 < T_2 \leq \infty,
\end{align}
and the $E_{p,q}(T_1,T_2)$ denote the set of all solenoidal distributions with the corresponding norm finite.
Here, the function spaces in \eqref{CL} are so-called the Chemin--Lerner spaces to be defined in Section \ref{sec:A}.
Here, the above result does not contain the end point case $p=\infty$, while Koch and Tataru \cite{Koc-Tat-01} proved the global well-posedness for small data in $BMO^{-1}(\mathbb{R}^3)$, which is the largest critical space ensuring the well-posedness of \eqref{eq:NS-1} ever known.
However, it was proven by \cites{Bou-Pav-08,Wan-15,Yon-10} that \eqref{eq:NS-1} is ill-posed in all end point critical Besov spaces $\dB_{\infty,q}^{-1}(\mathbb{R}^3)$ for all $1 \leq q \leq \infty$.
%See \cites{} for more results on well and ill-posedness of \eqref{eq:NS-1} in scaling critical spaces.

In this paper, we revisit the well-posedness in the scaling critical Besov spaces framework and show the global well-posedness for initial data whose the third exponent is arbitrary large and lies in the end point ill-posed Besov space $\dB_{\infty,\sigma}^{-1}(\mathbb{R}^3)$, provided that the first and second components of the initial data is sufficiently small in $\dB_{p,q}^{3/p-1}(\mathbb{R}^3)$ with $3/2 < p < 3$.
Let us state our main theorem precisely.
%The definition of function spaces in the following theorems are given in Appendix \ref{sec:A}.
\begin{thm}\label{thm:large}
    Let $p$, $q$, $r$, $\theta$, and $\sigma$ satisfy
    \begin{gather}
        \frac{3}{2} < p < 3,
        \qquad
        1 \leq \sigma \leq \infty,
        \qquad
        1 \leq q \leq 2\sigma,\\
        \qquad
        2 < r < \infty,
        \qquad
        0 < \theta < \frac{r}{r-1}\sp{1-\frac{3}{2p}}.
    \end{gather}
    Then, there exist positive constants $\eta=\eta(p,q,r,\theta,\sigma)$ and $C=C(p,q,r,\theta,\sigma)$ such that for any solenoidal initial data $a$ satisfying $a_{\rm h}=(a_1,a_2) \in \dB_{p,q}^{3/p-1}(\mathbb{R}^3)^2 $, $a_3 \in \dB_{\infty,\sigma}^{-1}(\mathbb{R}^3)$, and 
    \begin{align}\label{small}
        \n{a_{\rm h}}_{\dB_{p,q}^{\frac{3}{p}-1}}
        \exp \sp{C\n{a_3}_{\dB_{\infty,\sigma}^{-1}}^{\frac{r}{\theta}}}
        \leq \eta,
    \end{align}
    there exists a unique solution $u=(u_{\rm h},u_3)$ in the class 
    \begin{align}
        &
        u_{\rm h}=(u_1,u_2) \in \widetilde{C}\sp{[0,\infty) ; \dB_{p,q}^{\frac{3}{p}-1}(\mathbb{R}^3)}^2 \cap \widetilde{L^1}\sp{0,\infty ; \dB_{p,q}^{\frac{3}{p}+1}(\mathbb{R}^3)}^2, 
        \\
        &
        u_3 \in \widetilde{C}\sp{[0,\infty) ; \dB_{\infty,\sigma}^{-1}(\mathbb{R}^3)} \cap \widetilde{L^1}\sp{0,\infty ; \dB_{\infty,\sigma}^1(\mathbb{R}^3)}
    \end{align}
    with the estimates
    \begin{align}
        \n{u_{\rm h}}_{E_{p,q}(0,\infty)}
        \leq
        C\eta, 
        \qquad 
        \n{u_3}_{E_{\infty,\sigma}(0,\infty)}
        \leq
        C
        \n{a_3}_{\dB_{\infty,\sigma}^{-1}}
        +
        C\eta^2.
    \end{align}
\end{thm}
\begin{rem}
    Let us mention some remarks on our result.
    \begin{enumerate}
        \item 
        Chemin, Gallagher, and Paicu \cite{Che-Gal-Pai-11} constructed a unique global solution to \eqref{eq:NS-1} for initial data whose all components may be large in the weakest class $\dB_{\infty,\infty}^{-1}(\mathbb{R}^3)$.
        However, their initial data should be based on a small solenoidal vector field in some Gevrey class based on a $L^2$-Besov space.
        In comparison with this, our theorem do not need Gevrey condition, however we could not reach $\sigma=\infty$ and horizontal component need to be small. 
        \item 
        It was shown by Iwabuchi and Nakamura \cite{Iwa-Nak-13} that \eqref{eq:NS-1} is globally well-posed under the condition 
        \begin{align}
            \n{a}_{\dB_{\infty,\infty}^{-1}}
            \leq
            \frac{C_p}{\log\sp{e+\n{a}_{\dB_{p,\infty}^{\frac{3}{p}-1}}}}, \qquad 3 < p < \infty.
        \end{align}
        This result implies that the global well-posedness holds for large data in $\dB_{p,\infty}^{\frac{3}{p}-1}(\mathbb{R}^3)$ but small in $\dB_{\infty,\infty}^{-1}(\mathbb{R}^3)$.
        Compared with this, our advantage is that we need not assume the third component of the initial velocity is in the critical Besov spaces of finite integrability exponent and smallness in the end point class.
        \item 
        Since our proof is based on the Banach fixed point argument, we may readily see from our proof that the solution map $a \mapsto u$ from $\dB_{p,q}^{3/p-1}(\mathbb{R}^2)^2 \times \dB_{\infty,\sigma}^{-1}(\mathbb{R})$ to $E_{p,q}(0,T)^2 \times E_{\infty,\sigma}(0,T)$ is Lipschitz continuous,
        whereas \cites{Bou-Pav-08,Wan-15,Yon-10} proved the discontinuity of the solution map $\dB_{\infty,\sigma}^{-1}(\mathbb{R})^3 \to E_{\infty,\sigma}(0,T)^3$.
        \item 
        Let us give an example of solenoidal initial data whose horizontal components belongs to $\dB_{p,q}^{3/p-1}(\mathbb{R}^2)^2$ with all finite $p,q$, whereas the third component lies in $\dB_{\infty,\sigma}^{-1}(\mathbb{R}^3) \setminus \dB_{p,q}^{3/p-1}(\mathbb{R}^3)$.
        For $A,\delta>0$, $(k_1,k_2) \in \mathbb{R}^2 \setminus \{0\}$, and a solenoidal vector $\psi \in \mathscr{S}(\mathbb{R}^3)^3$ with $0 \notin \supp \widehat{\psi}$,
        we set
        \begin{align}
            a(x)
            :=
            \begin{pmatrix}
                0 \\ 0 \\ A \cos (k_1 x_1 + k_2 x_2)
            \end{pmatrix}
            + \delta \psi.
        \end{align}
        Then, we readily see that $\div a = 0$. 
        Moreover, since the $j$th Littlewood--Paley dyadic block $\Delta_j=\varphi_j*a$ is given by 
        \begin{align}
            \Delta_j a(x)
            =
            \begin{pmatrix}
                0 \\ 0 \\ A \widehat{\varphi_j}(k_1,k_2,0)\cos (k_1 x_1 + k_2 x_2)
            \end{pmatrix}
            +
            \delta 
            \Delta_j\psi,
        \end{align}
        it holds 
        \begin{align}
            \n{a_{\rm h}}_{\dB_{p,q}^{3/p-1}} \leq \delta\n{\psi_{\rm h}}_{\dB_{p,q}^{3/p-1}}
        \end{align}
        and
        \begin{align}
            c
            \frac{A}{\sqrt{k_1^2+k_2^2}}
            -
            \delta 
            \n{\psi_3}_{\dB_{\infty,\sigma}^{-1}}
            \leq 
            \n{a_3}_{\dB_{\infty,\sigma}^{-1}}
            \leq 
            C
            \frac{A}{\sqrt{k_1^2+k_2^2}}
            +
            \delta 
            \n{\psi_3}_{\dB_{\infty,\sigma}^{-1}}.
        \end{align}
        As $\cos (k_1x_1 + k_2x_2) \notin L^p(\mathbb{R}^3)$ unless $p=\infty$, we see that $a_3 \in \dB_{\infty,\sigma}^{-1}(\mathbb{R}^3) \setminus \dB_{p,q}^{3/p-1}(\mathbb{R}^3)$.
        Note that the Navier--Stokes flow from this initial data could not be constructed by known results.

        On the other hand, 
        the ill-posedness result in $\dB_{\infty,\sigma}^{-1}(\mathbb{R}^3)$ by \cites{Bou-Pav-08,Wan-15,Yon-10} considered the initial data given by 
        \begin{align}
            a_{Q,r}(x)
            =
            \frac{Q}{\sqrt{r}}
            \sum_{k=1}^r
            \begin{pmatrix}
                0 \\ \cos (h_kx_1) \\ \cos (h_kx_1 -x_2)
            \end{pmatrix}, 
            \qquad
            h_k=2^{k(k-1)/2}\gamma^{k-1}\eta.
        \end{align}
        Comparing this with our example, we see that the second component of $a_{Q,r}(x)$ plays the essential role for the ill-posed phenomenon.
    \end{enumerate}
\end{rem}
We explain the idea of the proof of Theorem \ref{thm:large}.
We first mention why we may choose the third component of the initial velocity from the ill-posed class $\dB_{\infty, \sigma}^{-1}(\mathbb{R}^3)$,
although it was shown that the bilinear estimates 
\begin{align}\label{est:bad}
    \n{\int_0^te^{(t-\tau)\Delta}\mathbb{P}(u(\tau) \cdot \nabla) v(\tau) d\tau}_{E_{\infty,\sigma}(0,T)}
    \leq
    C
    \n{u}_{E_{\infty,\sigma}(0,T)}
    \n{v}_{E_{\infty,\sigma}(0,T)}
\end{align}
{\it fails} for all $1 \leq \sigma \leq \infty$; See \cites{Bou-Pav-08,Wan-15,Yon-10}.
Using the relation $\partial_{x_3}u_3 = - \div_{\rm h} u_{\rm h} = -\nabla_{\rm h} \cdot  u_{\rm h}$ via the divergence free condition to rewrite the first equation of \eqref{eq:NS-1} as 
\begin{align}
    &
    \partial_t u_{\rm h} - \Delta u_{\rm h} + (u_{\rm h} \cdot \nabla_{\rm h}) u_{\rm h} +  u_3 \partial_{x_3}u_{\rm h} + \nabla_{\rm h}P = 0,  \label{eq:u_h}
    \\
    &
    \partial_t u_3 - \Delta u_3 + (u_{\rm h} \cdot \nabla_{\rm h}) u_3 -  u_3 \nabla_{\rm h} \cdot u_{\rm h} + \partial_{x_3}P = 0, \label{eq:u_3}
\end{align}
we may regard the equation for $u_3$ is linear, which enables us to show
the estimates of the following type:
\begin{align}\label{est:good}
    \n{\int_0^te^{(t-\tau)\Delta}(u_3 \nabla_{\rm h}\cdot u_{\rm h})(\tau) d\tau}_{E_{\infty,\sigma}(0,T)}
    \leq
    C
    \n{u_{\rm h}}_{E_{p,q}(0,T)}
    \n{v_3}_{E_{\infty,\sigma}(0,T)}
\end{align}
with $3/2 < p < 3$;
see Section \ref{sec:pre} below for more detail.
Since \eqref{eq:u_h}-\eqref{eq:u_3} has no nonlinear terms that is the second order with respect to $u_3$ and thus the method on \eqref{est:good} may close the all nonlinear estimates for $u_3$.
We also remark that the nonlinear term $u_3 \nabla_{\rm h} \cdot u_{\rm h}$ in \eqref{eq:u_3} cannot be written in the divergence form\footnote{The other nonlinear terms are able to be written in the divergence form. Therefore the estimates for them actually hold for $p<\infty$.}, which is the worst term in the view point of the regularity and requires the assumption $p<3$. 
For the relaxed smallness condition \eqref{small} is inspired by the seminal works \cites{Zha-09,Zha-09-E,Pai-Zha-11}, where they initially proposed the rewritten system \eqref{eq:u_h}-\eqref{eq:u_3} for the anisotropic Navier--Stokes equations\footnote{\eqref{eq:NS-1} with the viscosity replaced by horizontal Laplacian $\Delta_{\rm h}u=(\partial_{x_1}^2+\partial_{x_2}^2)u$.}, and proved the global well-posedness for initial data satisfying\footnote{The function space $\mathcal{B}^{0,\frac{1}{2}}(\mathbb{R}^3)$ is a kind of critical anisotropic Besov spaces introduced by \cite{Pai-05}.} $\n{a_{\rm h}}_{\mathcal{B}^{0,1/2}}\exp\sp{C\n{a_3}_{\mathcal{B}^{0,1/2}}^4} \ll 1$.
Their method was based on the energy estimates in some space-time Chemin--Lerner type spaces with some integrable weight function in time.
However, to our best effort applying their method to our aim, it works well only for the case of $p=2$ and $r=4$ and the functional framework for $u_3$ should be same as for $u_{\rm h}$.
To obtain Theorem \ref{thm:large}, we first prepare para-product estimates in Besov spaces; see Lemmas \ref{lemm:nonlin-1}, \ref{lemm:nonlin-2}, and \ref{lemm:nonlin-3} below. 
Then, we use not the time-weighted energy argument but the method of the decomposition of the time interval, which was proposed in the stability argument for large flow in the author's previous works \cites{Fuj-pre}; see Lemma \ref{lemm:time}.
This method enables us to extend the life span of the local solutions under some relaxed smallness condition \eqref{small}
since some space-time norm of $u_3$ will be sufficiently small in the each decomposed time interval.

This paper is organized as follows.
In Section \ref{sec:pre}, we prepare several useful tools in our analysis.
In Section \ref{sec:pf}, we present the proof of our main result.
Finally, in Appendix \ref{sec:A}, we summarize the definitions and basic properties of functions spaces that are used in this paper.

Throughout this paper, we denote by $C$ the constant, which may differ in each line. In particular, $C=C(a_1,...,a_n)$ means that $C$ depends only on $a_1,...,a_n$. 
For any integrability exponent $1\leq p \leq \infty$, we denote by $p'$ the H\"{o}lder conjugate of $p$. 
For two Banach spaces $X$ and $Y$ with $X \cap Y \neq \varnothing$, 
we write $\| \cdot \|_{X \cap Y} := \| \cdot \|_X + \| \cdot \|_Y$.
For $f=(f_1,...,f_n) \in X^n$, we set $\| f \|_{X}:= \| f_1\|_X + ... + \| f_n \|_X$.

\section{Key lemmas}\label{sec:pre}
In this section, we establish several lemmas that are used in the proof of our main results.
\subsection*{Time decomposition lemma}
To begin with, we introduce the time decomposition lemma that plays a key role in constructing global solution with large vertical velocity.
\begin{lemm}\label{lemm:time}
    Let $1 \leq p \leq \infty$, $1 \leq \sigma \leq r < \infty$, $s \in \mathbb{R}$, and $0<T\leq \infty$.
    For any $f \in \widetilde{L^r}(0,T;\dB_{p,\sigma}^s(\mathbb{R}^3))$ and $\varepsilon>0$, we define an integer 
    \begin{align}
        N := \left \lfloor \frac{\| f \|_{\widetilde{L^r}(0,T;\dB_{p,\sigma}^s)}^r}{\varepsilon^r} \right \rfloor,
    \end{align}
    where $\lfloor \cdot \rfloor$ denotes the floor function.
    Then, there exists a sequence of time $0=T_0<T_1<\dots<T_N<T_{N+1}=T$ such that 
    \begin{align}
        \n{f}_{\widetilde{L^r}(T_n,T_{n+1};\dB_{p,\sigma}^s)} \leq \varepsilon, \qquad n=0,\dots,N.
    \end{align}
\end{lemm}
\begin{proof}
    It suffices to consider the case of $N \geq 1$.
    Note that for $\tau \geq 0$, the function $[\tau,\infty) \ni t \mapsto \n{f}_{\widetilde{L^r}(\tau,t;\dB_{p,\sigma}^s)}$ is continuous due to $\sigma,r < \infty$.
    Let us define 
    \begin{align}
        T_n := 
        \begin{cases}
            0, & n=0, \\
            \sup\Mp{\tau >T_{n-1}\ ; \ \n{f}_{\widetilde{L^r}(T_{n-1},\tau;\dB_{p,\sigma}^s)} \leq \varepsilon}, & n=1,\dots,N, \\
            T, &n=N+1.
        \end{cases}
    \end{align}
    Then, we see that 
    \begin{align}
        \n{f}_{\widetilde{L^r}(T_n,T_{n+1};\dB_{p,\sigma}^s)} = \varepsilon, \qquad n=0,\dots,N-1.
    \end{align}
    It follows from the Minkowski inequality that 
    \begin{align}
        \| f \|_{\widetilde{L^r}(0,T;\dB_{p,\sigma}^s)}
        ={}&
        \Mp{
        \sum_{j \in\mathbb{Z}}
        \sp{
        2^{sj}\n{\Delta_j f}_{L^r(0,T;L^p)}
        }^{\sigma}
        }^{\frac{1}{\sigma}}
        \\
        ={}&
        \lp{
        \sum_{j \in\mathbb{Z}}
        \Mp{
        \sum_{n=0}^{N}
        \sp{
        2^{sj}\n{\Delta_j f}_{L^r(T_n,T_{n+1};L^p)}
        }^r
        }^{\frac{\sigma}{r}}
        }^{\frac{1}{\sigma}}
        \\
        \geq 
        {}&
        \lp{
        \sum_{n=0}^{N}
        \Mp{
        \sum_{j \in\mathbb{Z}}
        \sp{
        2^{sj}\n{\Delta_j f}_{L^r(T_n,T_{n+1};L^p)}
        }^\sigma
        }^{\frac{r}{\sigma}}
        }^{\frac{1}{r}}
        \\
        = 
        {}&
        \sp{
        \sum_{n=0}^{N}
        \| f \|_{\widetilde{L^r}(T_n,T_{n+1};\dB_{p,\sigma}^s)}^r
        }^{\frac{1}{r}}
        \\
        = 
        {}&
        \sp{
        N\varepsilon^r
        +
        \| f \|_{\widetilde{L^r}(T_N,T;\dB_{p,\sigma}^s)}^r
        }^{\frac{1}{r}},
    \end{align}
    which implies 
    \begin{align}
        \| f \|_{\widetilde{L^r}(0,T;\dB_{p,\sigma}^s)}^r
        \geq 
        N\varepsilon^r
        +
        \| f \|_{\widetilde{L^r}(T_N,\infty;\dB_{p,\sigma}^s)}^r.
    \end{align}
    Combining this with the fact $(N+1)\varepsilon^r \geq \| f \|_{\widetilde{L^r}(0,T;\dB_{p,\sigma}^s)}^r$ due to the definition of $N$, we obtain $\| f \|_{\widetilde{L^r}(T_N,T;\dB_{p,\sigma}^s)}^r \leq \varepsilon^r$.
    This completes the proof.
\end{proof}
\subsection*{Nonlinear estimates}
Next, we establish several nonlinear estiamtes.
%For $1 \leq p,q \leq \infty$ and $0<T\leq \infty$, we define 
%\begin{align}
%    E_{p,q}(0,T)
%    :=
%    \widetilde{C}\sp{[0,T) ; \dB_{p,q}^{\frac{3}{p}-1}(\mathbb{R}^3)} 
%    \cap 
%    \widetilde{L^1}\sp{0,T ; \dB_{p,q}^{\frac{3}{p}+1}(\mathbb{R}^3)}.
%\end{align}
\begin{lemm}\label{lemm:nonlin-1}
    Let $p$, $q$, $\sigma$, and $r$ satisfy
    \begin{align}
        1 < p < 3,
        \quad
        1 \leq q,\sigma \leq \infty,
        \quad
        2<r<\infty. 
    \end{align}
    Then, there exists a positive constant $C$ such that
    \begin{align}
        &
        \begin{aligned}
        &
        \n{\int_0^t e^{(t-\tau)\Delta}(u_{\rm h}(\tau) \cdot \nabla_{\rm h})v_{\rm h}(\tau) d\tau}_{E_{p,q}(0,T)}
        \\
        &\qquad
        \leq
        C
        \n{u_{\rm h}}_{\widetilde{L^r}(0,T;\dB_{p,q}^{\frac{3}{p}-1+\frac{2}{r}})}
        \n{v_{\rm h}}_{\widetilde{L^
        {r'}}(0,T;\dB_{p,q}^{\frac{3}{p}-1+\frac{2}{r'}})},
        \end{aligned}
        \\
        &
        \begin{aligned}
        &
        \n{\int_0^t e^{(t-\tau)\Delta}(u_{\rm h}(\tau) \cdot \nabla_{\rm h}v_3(\tau)) d\tau}_{E_{\infty,\sigma}(0,T)}
        \\
        &\qquad
        \leq
        C
        \n{u_{\rm h}}_{\widetilde{L^r}(0,T;\dB_{p,q}^{\frac{3}{p}-1+\frac{2}{r}})}
        \n{v_3}_{\widetilde{L^{r'}}(0,T;\dB_{\infty,\sigma}^{-1+\frac{2}{r'}})},
        \end{aligned}
        \\
        &
        \begin{aligned}
        &
        \n{\int_0^t e^{(t-\tau)\Delta}(u_3(\tau) \nabla_{\rm h} \cdot v_{\rm h}(\tau)) d\tau}_{E_{\infty,\sigma}(0,T)}
        \\
        &
        \qquad
        \leq
        C
        \n{u_3}_{\widetilde{L^{r}}(0,T;\dB_{\infty,\sigma}^{-1+\frac{2}{r}})}
        \n{v_{\rm h}}_{\widetilde{L^{r'}}(0,T;\dB_{p,q}^{\frac{3}{p}-1+\frac{2}{r'}})}
        \end{aligned}
    \end{align}
    for all $T>0$ and $u=(u_{\rm h},u_3) \in E_{p,q}(0,T)^2 \times E_{\infty,\sigma}(0,T)$.
\end{lemm}
\begin{proof}
It follows from Lemma \ref{lemm:max-reg} and Corollary \ref{cor:para} that
\begin{align}
    &
    \n{\int_0^t e^{(t-\tau)\Delta}(u_{\rm h}(\tau) \cdot \nabla_{\rm h})v_{\rm h}(\tau) d\tau}_{E_{p,q}(0,T)}
    \\
    &\qquad
    \leq{}
    C
    \n{(u_{\rm h} \cdot \nabla_{\rm h})v_{\rm h}}_{\widetilde{L^1}(0,T;\dB_{p,q}^{\frac{3}{p}-1})}\\
    &
    \qquad
    \leq{}
    C
    \n{u_{\rm h}}_{\widetilde{L^r}(0,T;\dB_{p,q}^{\frac{3}{p}-1+\frac{2}{r}})}
    \n{v_{\rm h}}_{\widetilde{L^
    {r'}}(0,T;\dB_{p,q}^{\frac{3}{p}-1+\frac{2}{r'}})}.
\end{align}
For the second estimate, we decompose the Duhamel integral by the para-product decomposition.
\begin{align}
    \int_0^t e^{(t-\tau)\Delta}(u_{\rm h}(\tau) \cdot \nabla_{\rm h}v_3(\tau)) d\tau
    =
    I_1[u,v](t) + I_2[u,v](t) + I_3[u,v](t),
\end{align}
where we have set
\begin{align}
    I_1[u,v](t)
    :={}&
    \sum_{m=1}^2
    \int_0^t e^{(t-\tau)\Delta} T_{u_m(\tau)}\partial_{x_m}v_3(\tau) d\tau,\\
    I_2[u,v](t)
    :={}&
    \sum_{m=1}^2
    \int_0^t e^{(t-\tau)\Delta} R(u_m(\tau),\partial_{x_m}v_3(\tau)) d\tau,\\
    I_3[u,v](t)
    :={}&
    \sum_{m=1}^2
    \int_0^t e^{(t-\tau)\Delta} T_{\partial_{x_m}v_3(\tau)}u_m(\tau) d\tau.
\end{align}
It follows from Lemmas \ref{lemm:max-reg} and \ref{lemm:nonlin} that 
\begin{align}
    \n{I_1[u,v]}_{E_{\infty,\sigma}(0,T)}
    \leq{}&
    C
    \sum_{m=1}^2
    \n{T_{u_m}\partial_{x_m}v_3}_{\widetilde{L^1}(0,T;\dB_{\infty,\sigma}^{-1})}\\
    \leq{}&
    C
    \sum_{m=1}^2
    \n{u_m}_{\widetilde{L^r}(0,T;\dB_{\infty,\infty}^{-1+\frac{2}{r}})}
    \n{\partial_{x_m}v_3}_{\widetilde{L^{r'}}(0,T;\dB_{\infty,\sigma}^{-2+\frac{2}{r'}})}\\
    \leq{}&
    C
    \n{u_{\rm h}}_{\widetilde{L^r}(0,T;\dB_{p,q}^{\frac{3}{p}-1+\frac{2}{r}})}
    \n{v_3}_{\widetilde{L^{r'}}(0,T;\dB_{\infty,\sigma}^{-1+\frac{2}{r'}})}.
\end{align}
Similarly, we have 
\begin{align}
    \n{I_3[u,v]}_{E_{\infty,\sigma}(0,T)}
    \leq{}
    C
    \n{u_{\rm h}}_{\widetilde{L^r}(0,T;\dB_{\infty,\infty}^{-1+\frac{2}{r}})}
    \n{v_3}_{\widetilde{L^{r'}}(0,T;\dB_{\infty,\sigma}^{-1+\frac{2}{r'}})}.
\end{align}
For the estimate of $I_2$, we see by Lemmas \ref{lemm:max-reg} and \ref{lemm:nonlin} that
\begin{align}
    \n{I_2[u,v]}_{E_{\infty,\sigma}(0,T)}
    \leq{}&
    C
    \n{I_2[u,v]}_{E_{p,\sigma}(0,T)}\\
    \leq{}&
    C
    \sum_{m=1}^2
    \n{R({u_m},\partial_{x_m}v_3)}_{\widetilde{L^1}(0,T;\dB_{p,\sigma}^{\frac{3}{p}-1})}\\
    \leq{}&
    C
    \sum_{m=1}^2
    \n{R({u_m},\partial_{x_m}v_3)}_{\widetilde{L^1}(0,T;\dB_{p,\sigma}^{\frac{3}{p}-1})}\\
    \leq{}&
    C
    \sum_{m=1}^2
    \n{u_m}_{\widetilde{L^r}(0,T;\dB_{p,\infty}^{\frac{3}{p}-1+\frac{2}{r}})}
    \n{\partial_{x_m}v_3}_{\widetilde{L^{r'}}(0,T;\dB_{\infty,\sigma}^{-2+\frac{2}{r'}})}\\
    \leq{}&
    C
    \n{u_{\rm h}}_{\widetilde{L^r}(0,T;\dB_{p,q}^{\frac{3}{p}-1+\frac{2}{r}})}
    \n{v_3}_{\widetilde{L^{r'}}(0,T;\dB_{\infty,\sigma}^{-1+\frac{2}{r'}})}.
\end{align}
The third one is obtained similarly as above.
Thus, we complete the proof.
\end{proof}

\begin{lemm}\label{lemm:nonlin-2}
    Let $p$, $q$, $r$, $\rho$, and $\sigma$ satisfy
    \begin{align}
        \frac{3}{2} < p < 3,
        \quad
        1 \leq q,\sigma \leq \infty,
        \quad
        1 \leq \rho < \frac{2p}{3},
        \quad
        2 < r < \infty. 
    \end{align}
    Then, there exists a positive constant $C$ such that
    \begin{align}
        &
        \n{\int_0^t e^{(t-\tau)\Delta}(u_3(\tau) \partial_{x_3} v_{\rm h}(\tau)) d\tau}_{E_{p,q}(0,T)}
        \\
        &\qquad
        \leq 
        C
        \n{u_3}_{\widetilde{L^{\rho}}(0,T;\dB_{\infty,\sigma}^{-1+\frac{2}{\rho}})}
        \n{v_{\rm h}}_{\widetilde{L^{\rho'}}(0,T;\dB_{p,q}^{\frac{3}{p}-1+\frac{2}{\rho'}})}
        \\
        &\quad
        \qquad
        +
        C
        \n{u_3}_{\widetilde{L^{r}}(0,T;\dB_{\infty,\sigma}^{-1+\frac{2}{r}})}
        \n{v_{\rm h}}_{\widetilde{L^{r'}}(0,T;\dB_{p,q}^{\frac{3}{p}-1+\frac{2}{r'}})}
    \end{align}
    for all $T>0$ and $u=(u_{\rm h},u_3) \in E_{p,q}(0,T)^2 \times E_{\infty,\sigma}(0,T)$.
\end{lemm}
\begin{proof}
    It follows from Lemma \ref{lemm:max-reg} that 
    \begin{align}
        \n{\int_0^t e^{(t-\tau)\Delta}(u_3(\tau) \partial_{x_3} u_{\rm h}(\tau)) d\tau}_{E_{p,q}(0,T)}
        \leq{}&
        C
        \n{u_3\partial_{x_3}u_{\rm h}}_{\widetilde{L^1}(0,T;\dB_{p,q}^{\frac{3}{p}-1})}.
    \end{align}
    Here, we decompose the product in the above right-hand side as 
    \begin{align}
        u_3\partial_{x_3}v_{\rm h}
        =
        T_{\partial_{x_3}v_{\rm h}}u_3
        +
        R(u_3,\partial_{x_3}v_{\rm h})
        +
        T_{u_3}\partial_{x_3}v_{\rm h}.
    \end{align}
    For the first term,
    we see that by Lemma \ref{lemm:nonlin} that 
    \begin{align}
        \n{T_{\partial_{x_3}v_{\rm h}}u_3}_{\widetilde{L^1}(0,T;\dB_{p,q}^{\frac{3}{p}-1})}
        \leq{}&
        C
        \n{\partial_{x_3}v_{\rm h}}_{\widetilde{L^{\rho'}}(0,T;\dB_{p,q}^{\frac{3}{p}-2+\frac{2}{\rho'}})}
        \n{u_3}_{\widetilde{L^{\rho}}(0,T;\dB_{\infty,\infty}^{-1+\frac{2}{\rho}})}
        \\
        \leq{}&
        C
        \n{v_{\rm h}}_{\widetilde{L^{\rho'}}(0,T;\dB_{p,q}^{\frac{3}{p}-1+\frac{2}{\rho'}})}
        \n{u_3}_{\widetilde{L^{\rho}}(0,T;\dB_{\infty,\infty}^{-1+\frac{2}{\rho}})},
    \end{align}
    where we have used 
    \begin{align}
        \frac{3}{p}-2+\frac{2}{\rho'}=\frac{3}{p}-\frac{2}{\rho} < 0.
    \end{align}
    For the other terms, Lemma \ref{lemm:nonlin} yields
    \begin{align}
        &
        \n{R(u_3,\partial_{x_3}v_{\rm h})}_{\widetilde{L^1}(0,T;\dB_{p,q}^{\frac{3}{p}-1})}
        +
        \n{T_{u_3}\partial_{x_3}v_{\rm h}}_{\widetilde{L^1}(0,T;\dB_{p,q}^{\frac{3}{p}-1})}
        \\
        &\quad
        \leq{}
        C
        \n{u_3}_{\widetilde{L^{r}}(0,T;\dB_{\infty,\infty}^{-1+\frac{2}{r}})}
        \n{\partial_{x_3}v_{\rm h}}_{\widetilde{L^{r'}}(0,T;\dB_{p,q}^{\frac{3}{p}-2+\frac{2}{r'}})}
        \\
        &\quad
        \leq{}
        C
        \n{u_3}_{\widetilde{L^{r}}(0,T;\dB_{\infty,\infty}^{-1+\frac{2}{r}})}
        \n{v_{\rm h}}_{\widetilde{L^{r'}}(0,T;\dB_{p,q}^{\frac{3}{p}-1+\frac{2}{r'}})},
    \end{align}
    Thus, we complete the proof.
\end{proof}
Next, we consider the pressure $P[u,v]$ defined by 
\begin{align}%\label{eq:decom-p}
        P[u,v]
        ={}&
        \sum_{\ell,m=1}^3
        \partial_{x_{\ell}}\partial_{x_m}(-\Delta)^{-1}
        (u_{\ell}v_m).
\end{align}
Let us decompose $P[u,v]$ as
\begin{align}\label{eq:decom-p}
        P[u,v]
        ={}&
        \sum_{\ell,m=1}^3
        \partial_{x_{\ell}}\partial_{x_m}(-\Delta)^{-1}
        (u_{\ell}v_m)
        \\
        ={}&
        \sum_{\ell,m=1}^2
        \partial_{x_{\ell}}\partial_{x_m}(-\Delta)^{-1}
        (u_{\ell}v_m)
        \\
        &
        +
        \sum_{m=1}^2
        \partial_{x_3}\partial_{x_m}(-\Delta)^{-1}
        (u_3v_m)
        +
        \sum_{m=1}^2
        \partial_{x_3}\partial_{x_m}(-\Delta)^{-1}
        (v_3u_m)
        \\
        &
        +
        \partial_{x_3}^2(-\Delta)^{-1}
        (u_3v_3)\\
        ={}&
        \sum_{\ell,m=1}^2
        \partial_{x_{\ell}}\partial_{x_m}(-\Delta)^{-1}
        (u_{\ell}v_m)
        \\
        &+
        \partial_{x_3}(-\Delta)^{-1}
        (v_{\rm h}\cdot \nabla_{\rm h} u_3)
        +
        \partial_{x_3}(-\Delta)^{-1}
        (u_{\rm h}\cdot \nabla_{\rm h} v_3)
        \\
        =:{}&
        P_1[u,v]
        +
        P_2[u,v].
\end{align}    
In the following lemma, we provide estimates of $P_1[u,v]$ and $P_2[u,v]$.
\begin{lemm}\label{lemm:nonlin-3}
    Let $p$, $q$, $r$, and $\sigma$ satisfy
    \begin{align}
        \frac{3}{2} < p < 3,
        \quad
        1 \leq \sigma \leq \infty,
        \quad
        1 \leq q \leq 2\sigma,
        \quad
        1 \leq \rho < \frac{2p}{3},
        \quad
        2 < r < \infty. 
    \end{align}
    Then, there exists a positive constant $C$ such that
    \begin{align}
        &
        \begin{aligned}
        &\n{\int_0^t e^{(t-\tau)\Delta}\nabla_{\rm h}P_1[u,v](\tau) d\tau}_{E_{p,q}(0,T)}
        \\
        &\quad
        \leq{}
        C
        \n{u_{\rm h}}_{\widetilde{L^{r}}(0,T;\dB_{p,q}^{\frac{3}{p}-1+\frac{2}{r}})}
        \n{v_{\rm h}}_{\widetilde{L^{r'}}(0,T;\dB_{p,q}^{\frac{3}{p}-1+\frac{2}{r'}})}
        \\
        &\qquad 
        +
        C
        \n{u_{\rm h}}_{\widetilde{L^{r'}}(0,T;\dB_{p,q}^{\frac{3}{p}-1+\frac{2}{r'}})}
        \n{v_{\rm h}}_{\widetilde{L^{r}}(0,T;\dB_{p,q}^{\frac{3}{p}-1+\frac{2}{r}})},
        \end{aligned}
        \\
        &
        \begin{aligned}
        &\n{\int_0^t e^{(t-\tau)\Delta}\partial_{x_3}P_1[u,v](\tau) d\tau}_{E_{\infty,\sigma}(0,T)}
        \\
        &\quad
        \leq{}
        C
        \n{u_{\rm h}}_{\widetilde{L^{r}}(0,T;\dB_{p,q}^{\frac{3}{p}-1+\frac{2}{r}})}
        \n{v_{\rm h}}_{\widetilde{L^{r'}}(0,T;\dB_{p,q}^{\frac{3}{p}-1+\frac{2}{r'}})}
        \\
        &\qquad
        +
        C
        \n{u_{\rm h}}_{\widetilde{L^{r'}}(0,T;\dB_{p,q}^{\frac{3}{p}-1+\frac{2}{r'}})}
        \n{v_{\rm h}}_{\widetilde{L^{r}}(0,T;\dB_{p,q}^{\frac{3}{p}-1+\frac{2}{r}})},
        \end{aligned}
    \end{align}
    and 
    \begin{align}
        &
        \begin{aligned}
        &
        \n{\int_0^t e^{(t-\tau)\Delta}\nabla_{\rm h}P_2[u,v](\tau) d\tau}_{E_{p,q}(0,T)}
        \\
        &\quad
        \leq{}
        C
        \n{v_3}_{\widetilde{L^{r}}(0,T;\dB_{\infty,\sigma}^{-1+\frac{2}{r}})}
        \n{u_{\rm h}}_{\widetilde{L^{r'}}(0,T;\dB_{p,q}^{\frac{3}{p}-1+\frac{2}{r'}})}\\
        &\qquad
        +
        C
        \n{u_{\rm h}}_{\widetilde{L^{r}}(0,T;\dB_{p,q}^{\frac{3}{p}-1+\frac{2}{r}})}
        \n{v_{\rm h}}_{\widetilde{L^{r'}}(0,T;\dB_{p,q}^{\frac{3}{p}-1+\frac{2}{r'}})}\\
        &\qquad
        +
        C
        \n{u_{\rm h}}_{\widetilde{L^{\rho'}}(0,T;\dB_{p,q}^{\frac{3}{p}-1+\frac{2}{\rho'}})}
        \n{v_3}_{\widetilde{L^{\rho}}(0,T;\dB_{\infty,\sigma}^{-1+\frac{2}{\rho}})},
        \end{aligned}
        \\
        &
        \begin{aligned}
        &
        \n{\int_0^t e^{(t-\tau)\Delta}\partial_{x_3}P_2[u,v](\tau) d\tau}_{E_{\infty,\sigma}(0,T)}
        \\
        &\quad
        \leq{}
        C
        \n{v_3}_{\widetilde{L^{r}}(0,T;\dB_{\infty,\sigma}^{-1+\frac{2}{r}})}
        \n{u_{\rm h}}_{\widetilde{L^{r'}}(0,T;\dB_{p,\infty}^{\frac{3}{p}-1+\frac{2}{r}})}\\
        &\qquad
        +
        C
        \n{u_{\rm h}}_{\widetilde{L^{\rho'}}(0,T;\dB_{p,q}^{\frac{3}{p}-1+\frac{2}{\rho'}})}
        \n{v_3}_{\widetilde{L^{\rho}}(0,T;\dB_{\infty,\sigma}^{-1+\frac{2}{\rho}})},
        \end{aligned}
    \end{align}
    for all $T>0$ and $u=(u_{\rm h},u_3) \in E_{p,q}(0,T)^2 \times E_{\infty,\sigma}(0,T)$.
\end{lemm}

\begin{proof}
    First, we consider the estimate of $P_1[u,v]$.
    By Lemmas \ref{lemm:max-reg} and \ref{lemm:nonlin}, there holds
    \begin{align}
        &\n{\int_0^t e^{(t-\tau)\Delta}\nabla_{\rm h}P_1[u,v](\tau) d\tau}_{E_{p,q}(0,T)}
        \\
        &\quad
        \leq{}
        C
        \sum_{\ell,m=1}^2 
        \n{u_\ell v_m}_{\widetilde{L^1}(0,T;\dB_{p,q}^{\frac{3}{p}})}
        \\
        &\quad
        \leq{}
        C
        \n{u_{\rm h}}_{\widetilde{L^{r}}(0,T;\dB_{p,q}^{\frac{3}{p}-1+\frac{2}{r}})}
        \n{v_{\rm h}}_{\widetilde{L^{r'}}(0,T;\dB_{p,q}^{\frac{3}{p}-1+\frac{2}{r'}})}
        \\
        &\qquad 
        \quad
        +
        C
        \n{u_{\rm h}}_{\widetilde{L^{r'}}(0,T;\dB_{p,q}^{\frac{3}{p}-1+\frac{2}{r'}})}
        \n{v_{\rm h}}_{\widetilde{L^{r}}(0,T;\dB_{p,q}^{\frac{3}{p}-1+\frac{2}{r}})}
    \end{align}
    and
    \begin{align}
        &\n{\int_0^t e^{(t-\tau)\Delta}\partial_{x_3}P_1[u,v](\tau) d\tau}_{E_{\infty,\sigma}(0,T)}
        \\
        &\quad
        \leq{}
        C
        \sum_{\ell,m=1}^2 
        \n{u_\ell u_m}_{\widetilde{L^1}(0,T;\dB_{p,\sigma}^{\frac{3}{p}})}
        \\
        &\quad
        \leq{}
        C
        \n{u_{\rm h}}_{\widetilde{L^{r}}(0,T;\dB_{p,2\sigma}^{\frac{3}{p}-1+\frac{2}{r}})}
        \n{v_{\rm h}}_{\widetilde{L^{r'}}(0,T;\dB_{p,2\sigma}^{\frac{3}{p}-1+\frac{2}{r'}})}
        \\
        &\qquad
        +
        C
        \n{u_{\rm h}}_{\widetilde{L^{r'}}(0,T;\dB_{p,2\sigma}^{\frac{3}{p}-1+\frac{2}{r'}})}
        \n{v_{\rm h}}_{\widetilde{L^{r}}(0,T;\dB_{p,2\sigma}^{\frac{3}{p}-1+\frac{2}{r}})}
        \\
        &\quad
        \leq{}
        C
        \n{u_{\rm h}}_{\widetilde{L^{r}}(0,T;\dB_{p,q}^{\frac{3}{p}-1+\frac{2}{r}})}
        \n{v_{\rm h}}_{\widetilde{L^{r'}}(0,T;\dB_{p,q}^{\frac{3}{p}-1+\frac{2}{r'}})}
        \\
        &\qquad
        +
        C
        \n{u_{\rm h}}_{\widetilde{L^{r'}}(0,T;\dB_{p,q}^{\frac{3}{p}-1+\frac{2}{r'}})}
        \n{v_{\rm h}}_{\widetilde{L^{r}}(0,T;\dB_{p,q}^{\frac{3}{p}-1+\frac{2}{r}})},
    \end{align}
    which provides the first estimate.

    For the estimate of $P_2[u,v]$, we may decompose it as
    \begin{align}
        P_2[u,v]
        ={}&
        \sum_{m=1}^2
        \partial_{x_3}(-\Delta)^{-1}
        R(u_m,\partial_{x_m}u_3)
        +
        \sum_{m=1}^2
        \partial_{x_3}(-\Delta)^{-1}
        T_{\partial_{x_m}v_3}u_m
        \\
        &
        +
        \sum_{m=1}^2
        \partial_{x_3}
        (-\Delta)^{-1}
        T_{u_m}\partial_{x_m}v_3\label{P2-1}
        \\
        ={}&
        \sum_{m=1}^2
        \partial_{x_3}(-\Delta)^{-1}
        R(u_m,\partial_{x_m}v_3)
        +
        \sum_{m=1}^2
        \partial_{x_3}(-\Delta)^{-1}
        T_{\partial_{x_m}v_3}u_m  
        \\
        &
        +
        \sum_{m=1}^2
        (-\Delta)^{-1}
        T_{u_m}\partial_{x_m}(-\nabla_{\rm h} \cdot v_{\rm h})
        +
        \sum_{m=1}^2
        (-\Delta)^{-1}
        T_{\partial_{x_3}u_m}\partial_{x_m}v_3. \label{P2-2}
    \end{align}
    Thus, we have by \eqref{P2-2} that
    \begin{align}
        &\n{\int_0^t e^{(t-\tau)\Delta}\nabla_{\rm h}P_2[u,v](\tau) d\tau}_{E_{p,q}(0,T)}
        \\
        &\quad
        \leq{}
        C\sum_{m=1}^2
        \sp{
        \n{R(u_m,\partial_{x_m}v_3)}_{\widetilde{L^1}(0,T;\dB_{p,q}^{\frac{3}{p}-1})}
        +
        \n{T_{\partial_{x_m}v_3}u_m}_{\widetilde{L^1}(0,T;\dB_{p,q}^{\frac{3}{p}-1})}
        }\\
        &\qquad
        +
        C\sum_{m=1}^2\n{T_{u_m}\partial_{x_m}(-\nabla_{\rm h}\cdot v_{\rm h})}_{\widetilde{L^1}(0,T;\dB_{p,q}^{\frac{3}{p}-2})}
        \\
        &\qquad
        +
        C\sum_{m=1}^2\n{T_{\partial_{x_3}u_m}\partial_{x_m}v_3}_{\widetilde{L^1}(0,T;\dB_{p,q}^{\frac{3}{p}-2})}
        \\
        &\quad
        \leq{}
        C
        \n{\nabla_{\rm h}v_3}_{\widetilde{L^{r}}(0,T;\dB_{\infty,\infty}^{-2+\frac{2}{r}})}
        \n{u_{\rm h}}_{\widetilde{L^{r'}}(0,T;\dB_{p,q}^{\frac{3}{p}-1+\frac{2}{r'}})}\\
        &\qquad
        +
        C
        \n{u_{\rm h}}_{\widetilde{L^{r}}(0,T;\dB_{\infty,\infty}^{-1+\frac{2}{r}})}
        \n{\nabla_{\rm h}^2v_{\rm h}}_{\widetilde{L^{r'}}(0,T;\dB_{p,q}^{\frac{3}{p}-3+\frac{2}{r'}})}\\
        &\qquad
        +
        C
        \n{\partial_{x_3}u_{\rm h}}_{\widetilde{L^{\rho'}}(0,T;\dB_{p,q}^{\frac{3}{p}-2+\frac{2}{\rho'}})}
        \n{\nabla_{\rm h}v_3}_{\widetilde{L^{\rho}}(0,T;\dB_{\infty,\infty}^{-2+\frac{2}{\rho}})}
        \\
        &\quad
        \leq{}
        C
        \n{v_3}_{\widetilde{L^{r}}(0,T;\dB_{\infty,\sigma}^{-1+\frac{2}{r}})}
        \n{u_{\rm h}}_{\widetilde{L^{r'}}(0,T;\dB_{p,q}^{\frac{3}{p}-1+\frac{2}{r'}})}\\
        &\qquad
        +
        C
        \n{u_{\rm h}}_{\widetilde{L^{r}}(0,T;\dB_{p,q}^{\frac{3}{p}-1+\frac{2}{r}})}
        \n{v_{\rm h}}_{\widetilde{L^{r'}}(0,T;\dB_{p,q}^{\frac{3}{p}-1+\frac{2}{r'}})}\\
        &\qquad
        +
        C
        \n{u_{\rm h}}_{\widetilde{L^{\rho'}}(0,T;\dB_{p,q}^{\frac{3}{p}-1+\frac{2}{\rho'}})}
        \n{v_3}_{\widetilde{L^{\rho}}(0,T;\dB_{\infty,\sigma}^{-1+\frac{2}{\rho}})},
    \end{align}
    and also see by \eqref{P2-1} that
    \begin{align}
        &
        \n{\int_0^t e^{(t-\tau)\Delta}\partial_{x_3}P_2[u,v](\tau) d\tau}_{E_{\infty,\sigma}(0,T)}
        \\
        &\quad
        \leq{}
        C\sum_{m=1}^2
        \n{R(u_m,\partial_{x_m}v_3)}_{\widetilde{L^1}(0,T;\dB_{p,\sigma}^{\frac{3}{p}-1})}
        \\
        &\qquad
        +
        C\sum_{m=1}^2
        \n{T_{\partial_{x_m}v_3}u_m}_{\widetilde{L^1}(0,T;\dB_{\infty,\sigma}^{-1})}
        +
        C\sum_{m=1}^2
        \n{T_{u_m}\partial_{x_m}v_3}_{\widetilde{L^1}(0,T;\dB_{\infty,\sigma}^{-1})}
        \\
        &\quad
        \leq{}
        C
        \n{\nabla_{\rm h}u_3}_{\widetilde{L^{r}}(0,T;\dB_{\infty,\sigma}^{-2+\frac{2}{r}})}
        \n{u_{\rm h}}_{\widetilde{L^{r'}}(0,T;\dB_{p,\infty}^{\frac{3}{p}-1+\frac{2}{r'}})}\\
        &\qquad
        +
        C
        \n{\nabla_{\rm h}u_3}_{\widetilde{L^{r}}(0,T;\dB_{\infty,\sigma}^{-2+\frac{2}{r}})}
        \n{u_{\rm h}}_{\widetilde{L^{r'}}(0,T;\dB_{\infty,\infty}^{-1+\frac{2}{r'}})}\\
        &\qquad
        +
        C
        \n{u_{\rm h}}_{\widetilde{L^{\rho'}}(0,T;\dB_{\infty,\infty}^{-1+\frac{2}{\rho'}})}
        \n{\nabla_{\rm h}u_3}_{\widetilde{L^{\rho}}(0,T;\dB_{\infty,\sigma}^{-2+\frac{2}{\rho}})}
        \\
        &\quad
        \leq{}
        C
        \n{v_3}_{\widetilde{L^{r}}(0,T;\dB_{\infty,\sigma}^{-1+\frac{2}{r}})}
        \n{u_{\rm h}}_{\widetilde{L^{r'}}(0,T;\dB_{p,\infty}^{\frac{3}{p}-1+\frac{2}{r}})}\\
        &\qquad
        +
        C
        \n{u_{\rm h}}_{\widetilde{L^{\rho'}}(0,T;\dB_{p,q}^{\frac{3}{p}-1+\frac{2}{\rho'}})}
        \n{v_3}_{\widetilde{L^{\rho}}(0,T;\dB_{\infty,\sigma}^{-1+\frac{2}{\rho}})},
    \end{align}
    which completes the proof.
\end{proof}

\section{Proof of Theorem \ref{thm:large}}\label{sec:pf}
In this section, we present the proof of Theorem \ref{thm:large}.
To begin with, we mention the notion of the solutions. 
Applying the Duhamel principle to \eqref{eq:u_h}-\eqref{eq:u_3} and the standard formula for the pressure of homogeneous incompressible flow, we may rewrite \eqref{eq:NS-1} as
\begin{align}
        u_{\rm h}(t)
        ={}&
        e^{t \Delta}a_{\rm h}
        -
        \int_0^t
        e^{(t-\tau)\Delta}
        \lp{
        (u_{\rm h} \cdot \nabla_{\rm h})u_{\rm h}
        +
        u_3 \partial_{x_3}u_{\rm h}
        +
        \nabla_{\rm h}P
        }(\tau)
        d\tau,
        \\
        u_3(t)
        ={}&
        e^{t \Delta}a_3
        -
        \int_0^t
        e^{(t-\tau)\Delta}
        \lp{
        u_{\rm h} \cdot \nabla_{\rm h}u_3
        -
        u_3
        \nabla_{\rm h} \cdot u_{\rm h} 
        +
        \partial_{x_3}P
        }(\tau)
        d\tau,
        \\
        P
        ={}&
        P[u,u]
        =
        \sum_{\ell,m=1}^3
        \partial_{x_{\ell}}\partial_{x_m}(-\Delta)^{-1}
        (u_{\ell}u_m)
\end{align}
with $\div u = 0$.
\begin{lemm}\label{lemm:LWP}
    Let $3/2 < p <3$, $1 \leq \sigma < \infty$, $1 \leq q \leq 2\sigma$.
    Let $a=(a_{\rm h}, a_3) \in \dB_{p,q}^{\frac{3}{p}-1}(\mathbb{R}^3)^2 \times \dB_{\infty,\sigma}^{-1}(\mathbb{R}^3)$ be a given solenoidal vector field.
    Then, there exists a time $T_0=T_0(p,q,\sigma,a)$ such that the system possesses a unique local solution
    \begin{align}
    &
    u_{\rm h} \in \widetilde{C}\sp{[0,T_0) ; \dB_{p,q}^{\frac{3}{p}-1}(\mathbb{R}^3)}^2 \cap \widetilde{L^1}\sp{0,T_0 ; \dB_{p,q}^{\frac{3}{p}+1}(\mathbb{R}^3)}^2, 
    \\
    &
    u_3 \in \widetilde{C}\sp{[0,T_0) ; \dB_{\infty,\sigma}^{-1}(\mathbb{R}^3)} \cap \widetilde{L^1}\sp{0,T_0 ; \dB_{\infty,\sigma}^1(\mathbb{R}^3)}.
    \end{align}
\end{lemm}
Lemma \ref{lemm:LWP} may be proved by the standard contraction mapping argument via the nonlinear estimates given by Lemmas \ref{lemm:nonlin-1}, \ref{lemm:nonlin-2}, and \ref{lemm:nonlin-3}.
Thus, we omit the precise proof.

Next, we consider the a priori estimate of the solution $u$ constructed in the above lemma.
\begin{lemm}\label{lemm:a-priori}
    Let $p$, $q$, $r$, $\theta$, and $\sigma$ satisfy the assumption of Theorem \ref{thm:large}.
    Then, there exists a positive constant $C_0=C_0(p,q,r,\theta,\sigma)$ such that
    \begin{align}
        &
        \begin{aligned}
        \n{u_{\rm h}}_{E_{p,q}(0,T)}
        \leq{}
        C_0
        \n{a_{\rm h}}_{\dB_{p,q}^{\frac{3}{p}-1}}
        &+
        C_0
        \n{u_{\rm h}}_{E_{p,q}(0,T)}^2
        \\
        &
        +
        C_0
        \n{u_{\rm h}}_{E_{p,q}(0,T)}
        \n{u_3}_{E_{\infty,\sigma}(0,T)}^{1-\theta}
        \n{u_3}_{\widetilde{L^{r}}(0,T;\dB_{\infty,\sigma}^{-1+\frac{2}{r}})}^{\theta}
        \end{aligned}
        \\
        &
        \begin{aligned}
        \n{u_3}_{E_{\infty,\sigma}(0,T)}
        \leq{}
        C_0
        \n{a_3}_{\dB_{\infty,\sigma}^{-1}}
        +
        C_0
        \n{u_{\rm h}}_{E_{p,q}(0,T)}^2
        +
        C_0
        \n{u_{\rm h}}_{E_{p,q}(0,T)}
        \n{u_3}_{E_{\infty,\sigma}(0,T)}
        \end{aligned}
    \end{align}
    for all $0 < T \leq T_{\rm max}$, where $T_{\rm max}$ denotes the maximal existence time.
\end{lemm}
\begin{proof}
    Let us choose $\rho$ so that
    \begin{align}
        \frac{1}{\rho} = \frac{1-\theta}{1} + \frac{\theta}{r}.
    \end{align}
    Then, we see that $1 < \rho < 2p/3$.
    Then, it follows from Lemmas \ref{lemm:nonlin-1}, \ref{lemm:nonlin-2}, and \ref{lemm:nonlin-3} that
    \begin{align}
        &
        \n{\int_0^t e^{(t-\tau)\Delta}(u_{\rm h}(\tau) \cdot \nabla_{\rm h})u_{\rm h}(\tau) d\tau}_{E_{p,q}(0,T)}
        \leq
        C
        \n{u_{\rm h}}_{E_{p,q}(0,T)}^2,
        \\
        &
        \n{\int_0^t e^{(t-\tau)\Delta}(u_3(\tau) \partial_{x_3} v_{\rm h}(\tau)) d\tau}_{E_{p,q}(0,T)}
        \leq
        C
        \n{u_{\rm h}}_{E_{p,q}(0,T)}
        \n{u_3}_{E_{\infty,\sigma}(0,T)},
        \\
        &
        \begin{aligned}
        \n{\int_0^t e^{(t-\tau)\Delta}\nabla_{\rm h}P(\tau) d\tau}_{E_{p,q}(0,T)}
        \leq{}&
        C
        \n{u_{\rm h}}_{E_{p,q}(0,T)}^2\\
        &
        +
        C
        \n{u_3}_{\widetilde{L^{\rho}}(0,T;\dB_{\infty,\sigma}^{-1+\frac{2}{\rho}})}
        \n{u_{\rm h}}_{E_{p,q}(0,T)}\\
        &
        +
        C
        \n{u_3}_{\widetilde{L^{r}}(0,T;\dB_{\infty,\sigma}^{-1+\frac{2}{r}})}
        \n{u_{\rm h}}_{E_{p,q}(0,T)},
        \end{aligned}
        \\
        &
        \n{\int_0^t e^{(t-\tau)\Delta}(u_{\rm h}(\tau) \cdot \nabla_{\rm h}u_3(\tau)) d\tau}_{E_{\infty,\sigma}(0,T)}
        \leq
        C
        \n{u_{\rm h}}_{E_{p,q}(0,T)}
        \n{u_3}_{E_{\infty,\sigma}(0,T)},
        \\
        &
        \n{\int_0^t e^{(t-\tau)\Delta}(u_3(\tau) \nabla_{\rm h} \cdot u_{\rm h}(\tau)) d\tau}_{E_{\infty,\sigma}(0,T)}
        \leq
        C
        \n{u_{\rm h}}_{E_{p,q}(0,T)}
        \n{u_3}_{E_{\infty,\sigma}(0,T)},
        \\
        &
        \n{\int_0^t e^{(t-\tau)\Delta}\partial_{x_3}P(\tau) d\tau}_{E_{\infty,\sigma}(0,T)}
        \leq{}
        C
        \n{u_{\rm h}}_{E_{p,q}(0,T)}^2
        +
        C
        \n{u_3}_{E_{\infty,\sigma}(0,T)}
        \n{u_{\rm h}}_{E_{p,q}(0,T)}.
    \end{align}
    Therefore, we have
    \begin{align}
        &
        \begin{aligned}
        \n{u_{\rm h}}_{E_{p,q}(0,T)}
        \leq{}
        C
        \n{a_{\rm h}}_{\dB_{p,q}^{\frac{3}{p}-1}}
        &+
        C
        \n{u_{\rm h}}_{E_{p,q}(0,T)}^2
        \\
        &
        +
        C
        \n{u_3}_{\widetilde{L^{\rho}}(0,T;\dB_{\infty,\sigma}^{-1+\frac{2}{\rho}})}
        \n{u_{\rm h}}_{E_{p,q}(0,T)}\\
        &
        +
        C
        \n{u_3}_{\widetilde{L^{r}}(0,T;\dB_{\infty,\sigma}^{-1+\frac{2}{r}})}
        \n{u_{\rm h}}_{E_{p,q}(0,T)},
        \end{aligned}
        \\
        &
        \begin{aligned}
        \n{u_3}_{E_{\infty,\sigma}(0,T)}
        \leq{}
        C
        \n{a_3}_{\dB_{\infty,\sigma}^{-1}}
        +
        C
        \n{u_{\rm h}}_{E_{p,q}(0,T)}^2
        +
        C
        \n{u_{\rm h}}_{E_{p,q}(0,T)}
        \n{u_3}_{E_{\infty,\sigma}(0,T)}.
        \end{aligned}
    \end{align}
    Combining them with  
    \begin{align}
        \n{u_3}_{\widetilde{L^{\rho}}(0,T;\dB_{\infty,\sigma}^{-1+\frac{2}{\rho}})}
        &\leq
        \n{u_3}_{\widetilde{L^1}(0,T;\dB_{\infty,\sigma}^{1})}^{1-\theta}
        \n{u_3}_{\widetilde{L^{r}}(0,T;\dB_{\infty,\sigma}^{-1+\frac{2}{r}})}^{\theta}\\
        &\leq 
        \n{u_3}_{E_{\infty,\sigma}(0,T)}^{1-\theta}
        \n{u_3}_{\widetilde{L^{r}}(0,T;\dB_{\infty,\sigma}^{-1+\frac{2}{r}})}^{\theta},\\
        %%%%%%%%
        \n{u_3}_{\widetilde{L^r}(0,T;\dB_{\infty,\sigma}^{-1+\frac{2}{r}})}
        &\leq
        \n{u_3}_{E_{\infty,\sigma}(0,T)}^{1-\theta}
        \n{u_3}_{\widetilde{L^{r}}(0,T;\dB_{\infty,\sigma}^{-1+\frac{2}{r}})}^{\theta},
    \end{align}
    we complete the proof.
\end{proof}

\begin{proof}[Proof of Theorem \ref{thm:large}]
    Let $p$, $q$, $r$, $\theta$, and $\sigma$ satisfy the assumption of Theorem \ref{thm:large}.
    We assume that 
    \begin{align}
        \n{a_{\rm h}}_{\dB_{p,q}^{\frac{3}{p}-1}}
        \exp
        \sp{2C_0(\log (2C_0))(4C_0)^{\frac{r}{\theta}}
        \sp{\frac{4}{3}C_0\n{a_3}_{\dB_{\infty,\sigma}^{-1}}+\frac{1}{12C_0}}^{\frac{r}{\theta}}
        }
        \leq
        \frac{1}{64C_0^4},
    \end{align}
    where $C_0$ is the positive constant appearing in Lemma \ref{lemm:a-priori}. 
    Let $u$ be the local solution to the system on the time interval $[0,T_{\rm max})$, where $T_{\rm max}$ denotes the maximal existence time.
    Let 
    \begin{align}
        T^*
        :=
        \sup
        \Mp{
        T \in (0,T_{\rm max})\ ;\ 
        \n{u_{\rm h}}_{E_{p,q}(0,T)}
        \leq
        \frac{1}{4C_0}
        }
    \end{align}
    and suppose by contradiction that 
    \begin{align}
        T^* < \infty.
    \end{align}
    Note that it holds
    \begin{align}
        \n{u_3}_{E_{\infty,\sigma}(0,T^*)}
        \leq
        C_0
        \n{a_3}_{\dB_{\infty,\sigma}^{-1}}
        +
        \frac{1}{16C_0}
        +
        \frac{1}{4}
        \n{u_3}_{E_{\infty,\sigma}(0,T^*)},
    \end{align}
    and thus we have 
    \begin{align}
        \n{u_3}_{E_{\infty,\sigma}(0,T^*)}
        \leq
        \frac{4}{3}C_0
        \n{a_3}_{\dB_{\infty,\sigma}^{-1}}
        +
        \frac{1}{12C_0}.
    \end{align}
    Let
    \begin{align}
        N:= 
        \left \lfloor 
        (4C_0)^{\frac{r}{\theta}}
        \n{u_3}_{E_{\infty,\sigma}(0,T^*)}^{\frac{r}{\theta}-r}
        \n{u_3}_{\widetilde{L^r}(0,T^*;\dB_{\infty,\sigma}^{-1+\frac{2}{r}})}^r
        \right \rfloor.
    \end{align}
    Then, by Lemma \ref{lemm:time}, there exists a sequence 
    \begin{align}
        0=T_0<T_1<\dots<T_N<T_{N+1}=T^*
    \end{align}
    of time such that 
    \begin{align}
        \n{u_3}_{\widetilde{L^r}(T_n,T_{n+1};\dB_{\infty,\sigma}^{-1+\frac{2}{r}})}^{\theta}
        \leq 
        \frac{1}{4C_0
        \n{u_3}_{E_{\infty,\sigma}(0,T^*)}^{1-\theta}} 
    \end{align}
    for $n = 0,1,\dots,N$.
    Applying Lemma \ref{lemm:a-priori} on $[T_n,T_{n+1}]$, we arrive at 
    \begin{align}
        &
        \n{u_{\rm h}}_{E_{p,q}(T_n,T_{n+1})}
        \leq
        C_0
        \n{u_{\rm h}(T_n)}_{\dB_{p,q}^{\frac{3}{p}-1}}
        +
        \frac{1}{4}
        \n{u_{\rm h}}_{E_{p,q}(T_n,T_{n+1})}
        +
        \frac{1}{4}
        \n{u_{\rm h}}_{E_{p,q}(T_n,T_{n+1})},
    \end{align}
    which yields 
    \begin{align}
        &
        \n{u_{\rm h}}_{E_{p,q}(T_n,T_{n+1})}
        \leq{}
        2C_0\n{u_{\rm h}(T_n)}_{\dB_{p,q}^{\frac{3}{p}-1}}.
    \end{align}
    Therefore, we see that
    \begin{align}
        \n{u_{\rm h}}_{E_{p,q}(T_n,T_{n+1})}
        &\leq{}
        2C_0
        \n{u_{\rm h}(T_n)}_{\dB_{p,q}^{\frac{3}{p}-1}}\\
        &\leq{}
        2C_0\n{u_{\rm h}}_{E_{p,q}(T_{n-1},T_{n})}\\
        &\leq{} \cdots \\
        &\leq{}
        (2C_0)^{n+1}\n{a_{\rm h}}_{\dB_{p,q}^{\frac{3}{p}-1}}, \qquad n = 1,\dots,N
    \end{align}
    and then
    \begin{align}
        &\n{u_{\rm h}}_{E_{p,q}(0,T^*)}
        \leq
        \sum_{n=0}^N
        \n{u_{\rm h}}_{E_{p,q}(T_n,T_{n+1})}\\
        &\quad\leq
        \sum_{n=0}^N
        (2C_0)^{n+1}\n{a_{\rm h}}_{\dB_{p,q}^{\frac{3}{p}-1}}\\
        &\quad\leq
        (2C_0)^{N+2}\n{a_{\rm h}}_{\dB_{p,q}^{\frac{3}{p}-1}}\\
        &\quad= 
        (2C_0)^2
        \n{a_{\rm h}}_{\dB_{p,q}^{\frac{3}{p}-1}}
        \exp
        \sp{(\log (2C_0))N}\\
        &\quad\leq 
        (2C_0)^2
        \n{a_{\rm h}}_{\dB_{p,q}^{\frac{3}{p}-1}}
        \exp
        \sp{(\log (2C_0))\left\lfloor(4C_0)^{\frac{r}{\theta}}
        \n{u_3}_{E_{\infty,\sigma}(0,T^*)}^{\frac{r}{\theta}}\right \rfloor}\\
        &\quad\leq 
        (2C_0)^3
        \n{a_{\rm h}}_{\dB_{p,q}^{\frac{3}{p}-1}}
        \exp
        \sp{2C_0(\log (2C_0))(4C_0)^{\frac{r}{\theta}}
        \sp{\frac{4}{3}C_0\n{a_3}_{\dB_{\infty,\sigma}^{-1}}+\frac{1}{12C_0}}^{\frac{r}{\theta}}
        }
        \\
        &\quad\leq
        \frac{1}{8C_0},
    \end{align}
    which contradicts the definition of $T^*$ and we have $T_{\rm max}=T^*=\infty$.
    This completes the proof.
\end{proof}

%******************
%Acknowledgements
%******************
\noindent
{\bf Data availability.} \\
Data sharing not applicable to this article as no datasets were generated or analysed during the current study.

\noindent
{\bf Conflict of interest.} \\
The author has declared no conflicts of interest.

\noindent
{\bf Acknowledgements.} \\
The author was supported by JSPS KAKENHI, Grant Number Navier--Stokes equations; well-posedness; critical Besov spaces.

\appendix
\def\thesection{\Alph{section}}
\section{Besov and Chemin--Lerner spaces}\label{sec:A}
Let $\mathscr{S}(\mathbb{R}^n)$ be the set of all Schwartz functions on $\mathbb{R}^n$, and 
we denote by {$\mathscr{S}'(\mathbb{R}^n)$} the set of all tempered distributions on $\mathbb{R}^n$.
Let $\varphi_0 \in \mathscr{S}(\mathbb{R}^n)$ satisfy 
\begin{align}
    \supp \widehat{\varphi_0} \subset \Mp{ \xi \in \mathbb{R}^n\ ;\ 2^{-1} \leq | \xi | \leq 2 },\quad
    0 \leq \widehat{\varphi_0}(\xi) \leq 1, 
\end{align}
and 
\begin{align}
    \sum_{j \in \mathbb{Z}}
    \widehat{\varphi_j}(\xi) = 1 \qquad {\rm for\ all\ }\xi \in  \mathbb{R}^n \setminus \{0\},
\end{align}
where we have set $\varphi_j(x):=2^{nj}\varphi_0(2^jx)$.
{Using them, we define the dyadic frequency localized operators by
\begin{align}
    \Delta_jf:=\mathscr{F}^{-1}\lp{\widehat{\varphi_j}(\xi)\widehat{f}}, \qquad f \in \mathscr{S}'(\mathbb{R}^n),\ j \in \mathbb{Z}.
\end{align}}
For $1 \leq p,\sigma \leq \infty$ and $s \in \mathbb{R}$, the Besov space $\dB_{p,\sigma}^s(\mathbb{R}^n)$ is defined as 
\begin{align}
    \dB_{p,\sigma}^s(\mathbb{R}^n)
    :={}&
    \Mp{
    f \in \mathscr{S}'(\mathbb{R}^n)/\mathscr{P}(\mathbb{R}^n)
    \ ; \ 
    \n{f}_{\dB_{p,\sigma}^s(\mathbb{R}^n)}<\infty
    },\\
    \n{f}_{\dB_{p,\sigma}^s(\mathbb{R}^n)}
    :={}&
    \n{
    \Mp{
    2^{sj}
    \n{\Delta_j f}_{L^p(\mathbb{R}^n)}
    }_{j \in \mathbb{Z}}
    }_{\ell^{\sigma}(\mathbb{Z})},
\end{align}
where $\mathscr{P}(\mathbb{R}^n)$ is the set of all polynomials on $\mathbb{R}^n$.
It is well-known that if $s <n/p$ or $(s,\sigma) = (n/p,1)$, then it holds
\begin{align}
    \dB_{p,\sigma}^s(\mathbb{R}^n)
    \sim
    \Mp{
    f \in \mathscr{S}'(\mathbb{R}^n)
    \ ; \ 
    \n{f}_{\dB_{p,\sigma}^s(\mathbb{R}^n)}<\infty,\quad
    f = \sum_{j \in \mathbb{Z}}\Delta_j f \quad {\rm in\ }\mathscr{S}'(\mathbb{R}^n)
    }.
\end{align}
For $1 \leq p,r,\sigma \leq \infty$, $s \in \mathbb{R}$, and an interval $I \subset \mathbb{R}$, we define 
the Chemin--Lerner space $\widetilde{L^r}(I;\dB_{p,\sigma}^s(\mathbb{R}^n))$ by 
\begin{align}
    \widetilde{L^r}(I;\dB_{p,\sigma}^s(\mathbb{R}^n))
    :={}&
    \Mp{
    F:I \to  \mathscr{S}'(\mathbb{R}^n)/\mathscr{P}(\mathbb{R}^n)
    \ ; \ 
    \n{F}_{\widetilde{L^r}(I;\dB_{p,\sigma}^s(\mathbb{R}^n))}
    <\infty
    },\\
    \n{F}_{\widetilde{L^r}(I;\dB_{p,\sigma}^s(\mathbb{R}^n))}
    :={}&
    \n{
    \Mp{
    2^{sj}
    \n{\Delta_j F}_{L^r(I;L^p(\mathbb{R}^n))}
    }_{j \in \mathbb{Z}}
    }_{\ell^{\sigma}(\mathbb{Z})}.
\end{align}
We also use the following notation
\begin{align}
    \widetilde{C}(I ; \dB_{p,\sigma}^s(\mathbb{R}^n))
    :=
    C(I ; \dB_{p,\sigma}^s(\mathbb{R}^n))
    \cap
    \widetilde{L^{\infty}}(I ; \dB_{p,\sigma}^s(\mathbb{R}^n)).
\end{align}
The Chemin--Lerner spaces were first introduced by \cite{Che-Ler-95} and continue to be frequently used for the analysis of compressible viscous fluids in critical Besov spaces.
The Chemin--Lerner spaces possess similar embedding properties as that for usual Besov spaces:
\begin{itemize}
    \item []
        $\widetilde{L^r}(I; \dB_{p,\sigma_1}^s(\mathbb{R}^n)) 
        \hookrightarrow
        \widetilde{L^r}(I; \dB_{p,\sigma_2}^s(\mathbb{R}^n))$ 
        for $1 \leqslant \sigma_1 \leqslant \sigma_2 \leqslant \infty$,
    \item []
        $\widetilde{L^r}(I; \dB_{p_1,\sigma}^{s+\frac{n}{p_1}}(\mathbb{R}^n)) 
        \hookrightarrow
        \widetilde{L^r}(I; \dB_{p_2,\sigma}^{s+\frac{n}{p_2}}(\mathbb{R}^n))$ 
        for $1 \leqslant p_1 \leqslant p_2 \leqslant \infty$.
\end{itemize}
It also holds by the Hausdorff--Young inequality that
\begin{align}
    &
    \widetilde{L^r}(I; \dB_{p,\sigma}^s(\mathbb{R}^n))
    \hookrightarrow
    {L^r}(I; \dB_{p,\sigma}^s(\mathbb{R}^n))
    \ {\rm for\ }1 \leqslant \sigma \leqslant r \leqslant \infty,\\
    %%%%%%%%%%%%%%%%%%
    &
    {L^r}(I; \dB_{p,\sigma}^s(\mathbb{R}^n)) 
    \hookrightarrow
    \widetilde{L^r}(I; \dB_{p,\sigma}^s(\mathbb{R}^n))
    \ {\rm for\ }1 \leqslant r \leqslant \sigma \leqslant \infty.
\end{align}
See \cite{Bah-Che-Dan-11} for more precise information of the Chemin--Lerner spaces.
One advantage of using the Chemin--Lerner spaces is that there holds the following maximal regularity estimates for the heat kernel $e^{t\Delta}$.
\begin{lemm}\label{lemm:max-reg}
    Let $1 \leq p, q \leq \infty$, $1 \leq \rho \leq r \leq \infty$, and $s \in \mathbb{R}$.
    Then, there exists a positive constant $C$ such that 
    \begin{align}
        \n{e^{t \Delta}a}_{\widetilde{L^r}(0,T;\dB_{p,q}^s)}
        &\leq
        C
        \n{a}_{\dB_{p,q}^s}, 
        \\
        \n{\int_0^t e^{(t-\tau)\Delta}f(\tau)d\tau}_{\widetilde{L^r}(0,T;\dB_{p,q}^{s+\frac{2}{r}})}
        &\leq
        C
        \n{f}_{\widetilde{L^{\rho}}(0,T;\dB_{p,q}^{s-2+\frac{2}{\rho}})}
    \end{align}
    for all $0<T\leq \infty$, $a \in \dB_{p,q}^s(\mathbb{R}^3)$ and $\widetilde{L^{\rho}}(0,T;\dB_{p,q}^{s-2+\frac{2}{\rho}}(\mathbb{R}^3))$.
\end{lemm}

To control the nonlinear terms in Chemin--Lerner spaces, 
we use we the para-product decomposition of the product $fg$ for two functions $f$ and $g$:
\begin{align}
    fg = T_fg + R(f,g) + T_gf,
\end{align}
where we have set 
\begin{align}
    T_fg := 
    \sum_{k \in \mathbb{Z}} 
    \left( \sum_{\ell \leq k-3} \Delta_{\ell} f \right) 
    \Delta_k g,
    \qquad
    R(f,g) :=
    \sum_{k \in \mathbb{Z}}
    \sum_{|k-\ell|\leq 2}
    \Delta_kf 
    \Delta_{\ell}g.
\end{align}
The the para-product estimates used in this paper are in order:
\begin{lemm}\label{lemm:nonlin}
    Let $1 \leq p_1,p_2,q,q_1,q_2,r,r_1,r_2 \leq \infty$ satisfy 
        \begin{align}
            \frac{1}{q}\leq\frac{1}{q_1}+\frac{1}{q_2}, 
            \qquad
            \frac{1}{r}=\frac{1}{r_1}+\frac{1}{r_2}.
        \end{align}
    Then, the following statements hold true.
    \begin{enumerate}
        \item 
        For any $s_1,s_2 \in \mathbb{R}$ satisfying 
        \begin{align}
            s_1 < \frac{3}{p_1}
        \end{align} 
        and $<$ above may be replaced by $\leq$ for the case of $q_1=1$,
        there exists a positive constant $C=C(p_1,p_2,q,q_1,q_2,s_1,s_2)$ such that
        \begin{align}
            \n{T_fg}_{\widetilde{L^r}(I;\dB_{p_2,q}^{s_1+s_2-\frac{3}{p_1}})}
            \leq
            C
            \n{f}_{\widetilde{L^{r_1}}(I;\dB_{p_1,q_1}^{s_1})}
            \n{g}_{\widetilde{L^{r_2}}(I;\dB_{p_2,q_2}^{s_2})}
        \end{align}
        for all intervals $I \subset \mathbb{R}$, $f \in \widetilde{L^{r_1}}(I;\dB_{p,q}^{s_1}(\mathbb{R}^3))$, and $g \in \widetilde{L^{r_1}}(I;\dB_{p,q}^{s_2}(\mathbb{R}^3))$.
        \item 
        For any $s_1,s_2 \in \mathbb{R}$ satisfying
        \begin{align}
            s_1+s_2 > \max\Mp{0, 3\sp{\frac{1}{p_1}+\frac{1}{p_2}-1}}
        \end{align}
        and $>$ above may be replaced by $\geq$ for the case of $q=\infty$ and $1/q_1+1/q_2=1$,
        there exists a positive constant $C=C(p_1,p_2,q,q_1,q_2,s_1,s_2)$ such that
        \begin{align}
            \n{R(f,g)}_{\widetilde{L^r}(I;\dB_{p_2,q}^{s_1+s_2-\frac{3}{p_1}})}
            \leq
            C
            \n{f}_{\widetilde{L^{r_1}}(I;\dB_{p_1,q_1}^{s_1})}
            \n{g}_{\widetilde{L^{r_2}}(I;\dB_{p_2,q_2}^{s_2})}
        \end{align}
        for all intervals $I \subset \mathbb{R}$, $f \in \widetilde{L^{r_1}}(I;\dB_{p,q}^{s_1}(\mathbb{R}^3))$, and $g \in \widetilde{L^{r_1}}(I;\dB_{p,q}^{s_2}(\mathbb{R}^3))$.
        \item 
        For any $s_1,s_2 \in \mathbb{R}$ satisfying
        \begin{align}
            s_2 < \min \Mp{\frac{3}{p_1},\frac{3}{p_2}}
        \end{align}
        and $<$ above may be replaced by $\leq$ for the case of $q_2=1$,
        there exists a positive constant $C=C(p_1,p_2,q,q_1,q_2,s_1,s_2)$ such that
        \begin{align}
            \n{T_gf}_{\widetilde{L^r}(I;\dB_{p_2,q}^{s_1+s_2-\frac{3}{p_1}})}
            \leq
            C
            \n{f}_{\widetilde{L^{r_1}}(I;\dB_{p_1,q_1}^{s_1})}
            \n{g}_{\widetilde{L^{r_2}}(I;\dB_{p_2,q_2}^{s_2})}
        \end{align}
        for all intervals $I \subset \mathbb{R}$, $f \in \widetilde{L^{r_1}}(I;\dB_{p,q}^{s_1}(\mathbb{R}^3))$, and $g \in \widetilde{L^{r_1}}(I;\dB_{p,q}^{s_2}(\mathbb{R}^3))$.
    \end{enumerate}
\end{lemm}
Lemma \ref{lemm:nonlin} may be proven by the para product argument in \cites{Bah-Che-Dan-11,Saw-18}; thus we omit the proof.
From this, we may immediately obtain the following corollary.
\begin{cor}\label{cor:para}
    Let $1 \leq p_1,p_2,q,q_1,q_2,r,r_1,r_2 \leq \infty$ satisfy 
        \begin{align}
            \frac{1}{q}\leq\frac{1}{q_1}+\frac{1}{q_2}, 
            \qquad
            \frac{1}{r}=\frac{1}{r_1}+\frac{1}{r_2}.
        \end{align}
    Let $s_1,s_2 \in \mathbb{R}$ satisfy
    \begin{align}
        s_1 < \frac{3}{p_1}, \quad
        s_2 < \min \Mp{\frac{3}{p_1},\frac{3}{p_2}},\quad
        s_1+s_2 > \max\Mp{0, 3\sp{\frac{1}{p_1}+\frac{1}{p_2}-1}}
    \end{align}
    Then, there exists a positive constant $C=C(p_1,p_2,q,q_1,q_2,s_1,s_2)$ such that
    \begin{align}
            \n{fg}_{\widetilde{L^r}(I;\dB_{p_2,q}^{s_1+s_2-\frac{3}{p_1}})}
            \leq
            C
            \n{f}_{\widetilde{L^{r_1}}(I;\dB_{p_1,q_1}^{s_1})}
            \n{g}_{\widetilde{L^{r_2}}(I;\dB_{p_2,q_2}^{s_2})}
    \end{align}
    for all intervals $I \subset \mathbb{R}$, $f \in \widetilde{L^{r_1}}(I;\dB_{p,q}^{s_1}(\mathbb{R}^3))$, and $g \in \widetilde{L^{r_1}}(I;\dB_{p,q}^{s_2}(\mathbb{R}^3))$.
\end{cor}

\end{document}